\documentclass{amsart}

\usepackage{amssymb,amsthm,graphicx,caption,mwe,float,adjustbox, comment, hyperref}%
\usepackage{amsmath,amsfonts,amsthm,amssymb,indentfirst,epic,xurl,graphics,needspace}

\newtheorem{theorem}{Theorem}[section]
\newtheorem{corollary}[theorem]{Corollary}

\newtheorem{lemma}[theorem]{Lemma}

\newtheorem{definition}[theorem]{Definition}

\title [Some answers regarding factorizations of finite groups]{Some answers regarding factorizations of finite groups}
\author[McCulloch]{Ryan McCulloch}
\address{Department of Mathematics and Statistics, Binghamton University, Binghamton, NY 13902; rmccullo1985@gmail.com}
\date{\today}

\begin{document}
	
\begin{abstract}
In this note we answer some questions of G. M. Bergman concerning factorizations of finite groups.  Such questions originated when M. H. Hooshmand asked whether, given a finite group $G$ and a factorization $|G| = n_1 \cdots n_k$, one can always find subsets $A_1, \dots , A_k$ of $G$ with $|A_i| = n_i$ such that $G = A_1 \cdots A_k$.  This was Question 19.35 of the Kourovka Notebook.  G. M. Bergman provided a counterexample for $k=3$, and Kabenyuk provided counterexamples for all $k \geq 3$.  We present these counterexamples with direct proofs that mirror Bergman's original $k=3$ argument, and we answer another recent question of Bergman.  This yields a new factorization result for a class of groups of order $p^r q$ where $p$ and $q$ are primes.
\end{abstract}


\maketitle

\section{Introduction}
We begin with a definition appearing in \cite{berg, hoosh1, sands}.

\begin{definition}
Suppose $G$ is a group, $k$ is a positive integer, and $A_1$, \dots , $A_k$ are subsets of $G$. In this situation, if the multiplication map $A_1 \times \cdots \times A_k \rightarrow G$ is bijective, we shall call the product $G = A_1 \cdots A_k$ a (k-fold) \textnormal{factorization} of $G$.
\end{definition}

If $G$ is finite, the above bijectivity condition can be reformulated as $G = A_1 \cdots A_k$ and $|G| = |A_1| \cdots |A_k|$.  This is how the condition is described in the Kourovka Notebook \cite{kour}, Question 19.35.  The question was raised by M. H. Hooshmand \cite{hoosh0, hoosh1}.  Also we shall declare that $1 < |A_i| < |G|$ for each $A_i$; see \cite{hoosh1} where this additional condition is explicitly stated.

The case $k=2$ is perhaps the most interesting case, because the question of whether every group $G$ admits a $2$-fold factorization for every factorization $(|A_1|,|A_2|)$ of its order is still open, see Question 20.37 in the Kourovka Notebook \cite{kour}.  Progress has been made and the question has been reduced to one for simple groups, see \cite{bild, hoosh1}.  

For $k=3$, Bergman \cite{berg} provided a counterexample for $G$ the alternating group $A_4$ and the factorization $(2,3,2)$ of $|G|$.  In that paper, Bergman asks whether there is a way to generalize the arguments to produce counterexamples for $k > 3$.  The question of whether there is any integer $k$ where we have $k$-fold factorizations existing for every finite group $G$ and every factorization of $|G|$ is Question VI in \cite{hoosh1}.  Kabenyuk  provided counterexamples for every $k \geq 3$; see Corollary 3.4 and Example 3.6 in  \cite{kab}.  We present the Frobenius group counterexamples and give a direct proof mirroring Bergman's proof for the $k=3$ case.

Additionally, we present a sufficient condition for $k$-factorizations of a finite group to exist.  This generalizes a recent result of Bergman on $2$-factorizations \cite{berg2}.  We obtain a corollary for a class of groups of order $p^r q$ where $p$ and $q$ are primes, establishing the applicability and novelty of our sufficient condition.

\section{The Example}

Lemma 2.1 in \cite{berg} will be crucial for our argument.  That lemma also has a third part regarding the sets $A_i$ for $1 < i < k$, however we do not need it and omit it in our statement below.  The proof is in \cite{berg}.

\begin{lemma}\label{lem: first_last}
Let $A_1 \cdots A_k$ be a factorization of a finite group $G$. Then
\begin{enumerate}
\item $|A_1|$ divides the order of the subgroup of $G$ generated by the set ${A_1}^{-1} A_1 = \{ g^{-1}h \,\, | \,\, g,h \in A_1 \}$, which can also be described as generated by any one of the subsets $g^{-1} A_1$ $(g \in A_1)$.  Moreover, that order is also the order of the subgroup generated by $A_1{A_1}^{-1} = \{ gh^{-1} \,\, | \,\, g,h \in A_1 \}$, equivalently, by any one of the subsets $A_1 g^{-1}$ $(g \in A_1)$.
\item Similarly, $|A_k|$ divides the order of the subgroup of $G$ generated by the set ${A_k}^{-1} A_k$, equivalently, by any of the subsets $g^{-1} A_k$ $(g \in A_k)$, and that order is also the order of the subgroup generated by $A_k{A_k}^{-1}$, equivalently, by any of the subsets $A_k g^{-1}$ $(g \in A_k)$.
\end{enumerate}
\end{lemma}

We also need the following lemma which appears in \cite{berg, sands}. 

\begin{lemma}\label{lem: e}
If $G = A_1 \cdots A_k$ is a factorization of a group $G$, then for all $g,h \in G$, $$(gA_1) \cdot A_2 \cdots A_{k-1} \cdot (A_k h)$$ is also a factorization of G.

Hence if for some positive integers $n_1, \dots, n_k$, $G$ has a $k$-fold factorization with $|A_i| = n_i$ $(i = 1, \dots, k)$, it has such a factorization in which $A_1$ and $A_k$ both contain the identity element $e$.
\end{lemma}

We are now ready to present the counterexample.  Let $k \geq 3$ and let $G$ be the Frobenius group of order $2^{k-1}(2^{k-1}-1)$ with kernel $N$ an elementary abelian $2$-group of order $2^{k-1}$ and complement $A$ cyclic of order $2^{k-1}-1$.  Viewing $N$ as the additive group of the field $\mathbb{F}_{2^{k-1}}$ and viewing $A$ as generated by a primitive element in the multiplicative group of the field, we have $A$ acting on $N$ via field multiplication.  Two key properties of this group that we need are that the action of $A$ on $N$ is transitive on the nonidentity elements of $N$, and that the elements of $G$ of even order are precisely the nonidentity elements of $N$, all of which have order $2$.

When $k=3$ the example is the alternating group $A_4$ and as we will see, it is the same example given by Bergman in \cite{berg}.  Our arguments mirror those for the $k=3$ case in \cite{berg}.

\begin{theorem}\label{thm1}
Let $k \geq 3$ and let $G$ be the Frobenius group of order $2^{k-1}(2^{k-1}-1)$ as described above.  Then $G$ possesses no $k$-fold factorization for the factorization $(2,\dots,2,2^{k-1}-1,2)$ of $|G|$.
\end{theorem}

\begin{proof}
Suppose $G = A_1 \cdots A_k$ is a factorization with $|A_i|=2$ for all $1 \leq i \leq k-2$, $|A_{k-1}| = 2^{k-1} - 1$, and $|A_k| = 2$.  By Lemma $\ref{lem: e}$ we can assume without loss of generality that $A_1$ and $A_k$ have the forms $\{e,a\}$ and $\{e,b\}$ respectively.  By Lemma \ref{lem: first_last}, $|A_1|$ divides $|\langle {A_1}^{-1}A_1 \rangle| = |\langle a \rangle|$ and $|A_k|$ divides $|\langle {A_k}^{-1}A_k \rangle| = |\langle b \rangle|$.  The only even order elements of $G$ are the elements of $N$ which are all order $2$, hence we conclude that $A_1$ and $A_k$ are in fact subgroups of $N$ (which may or may not be distinct).

Write $S = A_2 \cdots A_{k-1}$.  Note that for each $g \in S$ we have $A_1gA_k \subseteq NgN = Ng = gN$ as $N$ is a normal subgroup of $G$.  And so in order for $G = A_1 \cdot S \cdot A_k$ to hold, we must have $S$ containing representatives from every coset of $N$ in $G$.  As the action of the complement $A$ is transitive on the nonidentity elements of $N$, there exists $g \in S$ such that $g$ conjugates $A_1$ to $A_k$.

Hence $A_1 g A_k = g A_k A_k$.  But the multiplication map $A_k \times A_k \rightarrow A_k$ is not one-to-one; and it follows that the multiplication map $A_1 \times \cdots \times A_k \rightarrow G$ is not one-to-one, contradicting $G = A_1 \cdots A_k$ being a factorization.
\end{proof}

\section{A Sufficient Condition for $k$-Factorizations to Exist}

In a recent note by Bergman \cite{berg2}, a sufficient condition is given for $2$-factorizations of a finite group to exist.  Bergman asks whether a similar approach can be used to give a sufficient condition for $k$-factorizations to exist for larger $k$.  Our theorem below is a direct generalization of Bergman's Theorem 1.1 in \cite{berg2}, and the argument mirrors his argument.  The $k=2$ case yields Bergman's result.

\begin{theorem}\label{thm2}
Let $G$ be a finite group, and
\begin{equation}\begin{minipage}[c]{35pc}\label{d.G0Gn}
$\{e\}\,=\,G_0\ <\ G_1\ <\ \dots\ <\ G_n\,=\,G$ \ $(n\geq 0)$
\end{minipage}\end{equation}
a chain of subgroups of $G$.

Then for any factorization $|G| = a_1 a_2 \cdots a_k$ with $a_1, \dots, a_k$ positive integers, a {\em sufficient} condition for there
to exist subsets $A_1,\dots,A_k \subseteq G$ of cardinalities $|A_i| = a_i$ for all $i$ such that $$G = A_1 \cdots A_k$$
is that there exists a partition $\{ s_1 , \dots, s_k \}$ of $\{1,\dots,n\}$ such that 
\begin{equation*}\begin{minipage}[c]{35pc}
for each $j \in \{1,\dots,n\}$, if $j \in s_t$, then $s_a \cap \{1,\dots,j\} = \emptyset$ holds either for all $a < t$ or for all $a > t$,
\end{minipage}\end{equation*}
and such that for each $t \in \{1,\dots,k\}$, $\prod_{i\in s_t}\,|G_i:\,G_{i-1}| = a_t$.
\end{theorem}

\begin{proof}
Given a finite group $G$, a chain~\eqref{d.G0Gn}, and such subsets $s_1, \dots s_k$, we shall show by induction on $j$ that for each $j \leq n$, 
\begin{equation}\begin{minipage}[c]{35pc}\label{d.A1jA2j}
there exists a factorization $G_j = A_{1,j} \cdots A_{k,j}$, such that\\
 for each $t \in \{1,\dots,k\}$, $|A_{t,j}| = \Pi_{i \in s_t \cap \{1,\dots,j\}} | G_i : G_{i-1}|$.
\end{minipage}\end{equation}

This is trivially true for $j=0$.   (In that case the numerical products shown are over the empty set, hence
their values are all $1 = |\{e\}|$, and $G_0 = \{e\} = \{e\}\cdots\{e\}$.)

So let $j < n$, and assume inductively that we have a factorization $G_j = A_{1,j} \cdots A_{k,j}$ as in~\eqref{d.A1jA2j}.  

Now $j+1 \in s_t$ for some unique $t$.  We know that $s_a \cap \{1,\dots,j+1\} = \emptyset$ holds either for all $a < t$ or for all $a > t$.  Without loss of generality, suppose 
\begin{equation*}\begin{minipage}[c]{35pc}
$s_a \cap \{1,\dots,j+1\} = \emptyset$ for all $a < t$.
\end{minipage}\end{equation*}
Thus $|A_{a,j}| = 1$ for all $a < t$.  Write $S = A_{1,j} \cdots A_{t-1,j}$.  So 
\begin{equation}\begin{minipage}[c]{35pc}\label{d.i_in_s1}
$G_j = S A_{t,j} \cdots A_{k,j}$, and by Lemma \ref{lem: e}, we may assume $S = \{ e \}$.
\end{minipage}\end{equation}

Let $C_{j+1}$ be a full set of representatives of the left cosets of $G_j$ in $G_{j+1}$, and let
$A_{t,j+1} = C_{j+1} A_{t,j}$ and let $A_{i,j+1} = A_{i,j}$ for all $i \neq t$.  

Given that $A_{1,j} \dots A_{k,j}$have the cardinalities indicated in~\eqref{d.A1jA2j}, we see that the new sets $A_{1,j+1} \dots A_{k,j+1}$ will
have cardinalities as in the $j\,{+}\,1$ case of~\eqref{d.A1jA2j}.  Finally, assuming~\eqref{d.i_in_s1}, we obtain $$A_{1,j+1} \cdots A_{k,j+1} = C_{j+1} A_{t,j} \cdots A_{k,j} = C_{j+1} G_j = G_{j+1}.$$

So we have proved the $j + 1$ case of~\eqref{d.A1jA2j};
so by induction,~\eqref{d.A1jA2j} holds for all $j \leq n$. The case $j = n$ is
our desired conclusion.
\end{proof}

Note that when $k=2$, the condition that
\begin{equation*}\begin{minipage}[c]{35pc}
for each $j \in \{1, \dots, n\}$, if $j \in s_t$, then $s_a \cap \{1,\dots,j\} = \emptyset$ holds either for all $a < t$ or for all $a > t$
\end{minipage}\end{equation*}
 will be true vacuously, as $t$ is either $1$ or $2$ and so either there are no $a < t$ or there are no $a > t$.  Hence for the $k=2$ case we have no additional restrictions on the sets $s_1$ and $s_2$, and we obtain Bergman's Theorem 1.1 of \cite{berg2}.

It is known that if $G$ is solvable, then a $2$-factorization exists for every factorization $(a,b)$ of $|G|$, see \cite{hoosh1,bild}.  Theorem \ref{thm2} (and Theorem 1.1 in \cite{berg2}) provide a new proof.

\begin{corollary}\label{solv}
 If $G$ is a solvable finite group, then a $2$-factorization of $G$ can be realized via Theorem \ref{thm2} for every factorization $(a,b)$ of $|G|$.
\end{corollary}

\begin{proof}
As $G$ is solvable, $G$ has a composition series where each factor group is cyclic of prime order.  Thus for any factorization $(a,b)$ of $|G|$, $a$ is a product of prime indices in its composition series, and so is $b$.
\end{proof}

We close this section with a corollary regarding the alternating group $A_4$.

\begin{corollary}
A $k$-factorization of the alternating group $A_4$ can be realized via Theorem \ref{thm2} for every factorization of $|A_4|$ except for $(2,3,2)$, for which no $3$-factorization exists.
\end{corollary}

\begin{proof}
That no $3$-factorization of $A_4$ exists for $(2,3,2)$ is Bergman's original counterexample, and follows from Theorem \ref{thm1}.  As $A_4$ is solvable, by Corollary \ref{solv}, all possible $2$-factorizations are realized via Theorem \ref{thm2}.  The remaining factorizations to realize via Theorem \ref{thm2} are $(2,2,3)$ and $(3,2,2)$.  Take the composition series $$G_0 < G_1 < G_2 < G_3,$$ where $\{e\} = G_0$, $\mathbb{Z}_2 = G_1$, $\mathbb{Z}_2 \times \mathbb{Z}_2 = G_2$, and $A_4 = G_3$.  For the $(2,2,3)$ case, let $s_1 = \{1\}$, $s_2 = \{2\}$, and $s_3 = \{3\}$.  One sees that $s_3 \cap \{1,2\} = \emptyset$; and for each $t \in \{1,2,3\}$, $\prod_{i\in s_t}\,|G_i:\,G_{i-1}| = a_t$ holds.  For the $(3,2,2)$ case, let $s_1 = \{3\}$, $s_2 = \{2\}$, and $s_3 = \{1\}$.   In that case, one sees that $s_1 \cap \{1,2\} = \emptyset$; and for each $t \in \{1,2,3\}$, $\prod_{i\in s_t}\,|G_i:\,G_{i-1}| = a_t$ holds. 
\end{proof}

\section{Applicability and Novelty}

Lemma 1.2 in \cite{hoosh1} asserts that a finite group $G$ of order $n$ admits a $k$-factorization for every factorization of $n$ of size $k$ provided that $G$ admits a $k$-factorization for every prime factorization of $n$.

We restate Theorem \ref{thm2} for the situation where $k$ is the number of prime divisors (with multiplicity) of $|G|$.  The partition $s_1,\dots,s_k$ in the statement of Theorem \ref{thm2} now consists of singletons, and we recast the sufficient conditions in terms of a permutation.

\begin{corollary}\label{cor4}
Let $G$ be a finite group of order $n$, and
\begin{equation*}\begin{minipage}[c]{35pc}
$\{e\}\,=\,G_0\ <\ G_1\ <\ \dots\ <\ G_k\,=\,G$ \ $(k\geq 0)$
\end{minipage}\end{equation*}
a chain of subgroups of $G$ where $k$ is the number of prime divisors (with multiplicity) of $|G|$.

Then for any prime factorization $(a_1, \dots, a_k)$ of $n$, a {\em sufficient} condition for there
to exist subsets $A_1,\dots,A_k \subseteq G$ of cardinalities $|A_i| = a_i$ for all $i$ such that $$G = A_1 \cdots A_k$$
is that there exists a permutation $\sigma$ of $\{1,\dots,k\}$ such that 
\begin{equation*}\begin{minipage}[c]{35pc}
for each $j \in \{1,\dots,k\}$, if $\sigma(t)=j$, then $\sigma(a) > j$ holds either for all $a < t$ or for all $a > t$,
\end{minipage}\end{equation*}
and such that for each $t \in \{1,\dots,k\}$, $ |G_{\sigma(t)}:\,G_{\sigma(t)-1}| = a_t$.
\end{corollary}

We apply Corollary \ref{cor4} to a class of groups of order $p^r q$ where $p$ and $q$ are primes.

\begin{corollary}\label{cor5}
Let $G$ be a finite group and $N$ a normal subgroup of $G$ so that $|N| = p^m q$, $|G : N| = p^a$ where $p$ and $q$ are primes, $m$ and $n$ are non-negative integers, and $m \leq a +1$.  Then $G$ admits a $k$-factorization for every factorization of $|G|$ of size $k$.
\end{corollary}

\begin{proof}
By Lemma 1.2 in \cite{hoosh1}, the result follows if we exhibit $k$-factorizations for every prime factorization $(a_1,\dots,a_k)$ of $n=|G|$.  Let $(a_1,\dots,a_k)$ be any such prime factorization of $n$, and so $k=m+a+1$.  Note that a Sylow $p$-subgroup of $N$ possesses a chain of subgroups, with each factor in the chain having index $p$.  The quotient group $G/N$, being a $p$-group, also possesses this property.  Thus we obtain a chain
\begin{equation*}\begin{minipage}[c]{35pc}
$\{e\}\,=\,G_0\ <\ G_1\ <\ \dots\ <\ G_k\,=\,G$ \ 
\end{minipage}\end{equation*}
where $|G_i : G_{i-1}| = p$ for all $i$ with $1 \leq i \leq m$, $|G_{m+1} : G_m| = q$ with $G_{m+1} = N$, and $|G_i : G_{i-1}| = p$ for all $i$ with $m+2 \leq i \leq k$.

Now $a_{j_0} = q$ for some $j_0 \in \{ 1, \dots, k \}$ and $a_j = p$ for all $j \neq j_0$.  

If $j_0 = m+1$, let $\sigma$ be the identity permutation on $\{ 1, \dots, k \}$.  Trivially, for each $j \in \{ 1, \dots, k \}$, $\sigma(a) = a > j$ holds for all $a > j$; and we have $|G_j : G_{j-1}| = a_j$ for all $j$, and the result follows by Corollary \ref{cor4}.

If $j_0 < m+1$, define $\sigma$ as follows:  
\begin{equation*}\begin{minipage}[c]{35pc}
$\sigma(i) = k-i+1$ \ \ \ for $1 \leq i < j_0$ \\ 
$\sigma(j_0) = m+1$ \\ 
$\sigma(i) = i-j_0$ \ \ \ \ \ \ \ \ for $j_0 < i \leq j_0 + m$ \\ 
$\sigma(i) = i-j_0+1$ \ \ \ for $j_0 + m < i \leq k$ .\
\end{minipage}\end{equation*}

One readily verifies that $\sigma$ is a permutation of $\{1,\dots,k\}$: if $j_0 < m+1$, then for $i < j_0$ we have $i \leq m$, and we obtain $$k - i + 1 = m+a+1 -i + 1 \geq m+a+1-m+1 = a+2 > a+1 \geq m,$$ and so one sees where the condition $m \leq a+1$ is needed.  If $t < j_0$, then for all $a < t$, $\sigma(a) = k-a+1 > k-t+1 = \sigma(t)$.  If $j_0 < t \leq j_0 + m$, then for all $a > t$, $\sigma(t) = t-j_0 < a - j_0 \leq \sigma(a)$.  If $j_0 + m < t \leq k$, then for all $a > t$, $\sigma(t) = t-j_0 +1 < a - j_0 + 1 = \sigma(a)$.  If $t=j_0$, then for all $a < t$, $\sigma(a) = k-a+1 > m+1 = \sigma(t)$, since $k = m+a+1$.

Finally, $$|G_{\sigma(j_0)} : G_{\sigma(j_0)-1}| = |G_{m+1} : G_m| = q = a_{j_0},$$ and $$|G_{\sigma(j)} : G_{\sigma(j)-1}| = p = a_{j}$$ for all $j \neq j_0$, and the result follows by Corollary \ref{cor4}.

If $j_0 > m+1$, define $\sigma$ as follows:  
\begin{equation*}\begin{minipage}[c]{35pc}
$\sigma(i) = j_0-i+1$ \ \ \ for $1 \leq i < j_0 - m$ \\ 
$\sigma(i) = j_0 - i$ \ \ \ \ \ \ \ \ \ for $j_0 - m \leq i < j_0$ \\ 
$\sigma(j_0) = m+1$ \\ 
$\sigma(i) = i$ \ \ \ \ \ \ \ \ \ \ \ \ \ \ \ for $j_0 < i \leq k$ .\
\end{minipage}\end{equation*}

One readily verifies that $\sigma$ is a permutation of $\{1,\dots,k\}$: in this situation it is clear; all that is needed is that $j_0 > m+1$.  If $1 < t \leq j_0 - m$, then for all $a < t$, $\sigma(t) = j_0 - t + 1 < j_0 - a + 1 = \sigma(a)$.  If $j_0 - m \leq t < j_0$, then for all $a < t$, $\sigma(t) = j_0 - t < j_0 - a \leq \sigma(a)$.  If $j_0 < t \leq k$, then for all $a > t$, $\sigma(a) = a > t = \sigma(t)$.  If $t=j_0$, then for all $a > t$, $\sigma(a) = a > t = j_0 > m+1 = \sigma(t)$.

Finally, $$|G_{\sigma(j_0)} : G_{\sigma(j_0)-1}| = |G_{m+1} : G_m| = q = a_{j_0},$$ and $$|G_{\sigma(j)} : G_{\sigma(j)-1}| = p = a_{j}$$ for all $j \neq j_0$, and the result follows by Corollary \ref{cor4}.
\end{proof}

Every finite supersolvable group is $k$-factorizable for every factorization of $|G|$ of size $k$; see Theorem 2.1 in \cite{hoosh1}.  Note that Corollary \ref{cor5} includes non-supersolvable groups.  The group $S_4$ is included: $p=2$,$q=3$,$m=2$, and $a=1$ with $N$ the normal $A_4$ subgroup.  It is known that $S_4$ is $4$-factorizable for every prime factorization of $24$; see Section 5 of \cite{kab}.  The techniques used for $S_4$ and for other non-supersolvable groups are more complicated and often ad hoc.  Note that the subgroup $A_4$ of $S_4$ does not admit a $3$-factorization for $(2,3,2)$, and so a basic inductive argument does not suffice in handling $S_4$.  It is worth noting that up to order 100, only eight groups fail to be $k$-factorizable for some $k$ and some factorization of $|G|$ of size $k$, see \cite{kab}.  

We provide another non-trivial, non-supersolvable group handled by Corollary \ref{cor5}.  We verified the details using GAP \cite{GAP}; the group has order $1536 = 2^9 \cdot 3$.  Let $V$ be a $5$-dimensional vector space over $\mathbb{F}_2$.  One finds a subgroup $C$ of $\mathrm{GL}(5,2)$ of order 3 so that a Sylow $2$-subgroup $P$ of $N_{\mathrm{GL}(5,2)}(C)$ has order $16$.   Let $G = [V]CP$.  Then $N = [V]C$ is a normal subgroup of $G$ of order $2^5 \cdot 3$ and $|G : N| = 2^4$.  So $p=2$,$q=3$,$m=5$, and $a=4$ and Corollary \ref{cor5} applies.  In conclusion, Corollary \ref{cor5} can handle non-supersolvable groups, many of which previously could only have been handled through ad hoc arguments, and thus we establish the applicability and novelty of Theorem \ref{thm2}.

\section{Acknowledgment}
The author thanks George Bergman for helpful discussion and for providing motivation for this paper.

\end{document}